\UseRawInputEncoding

\documentclass{amsart}
\usepackage{amsmath,amsthm,amssymb}
\usepackage[truedimen,margin=25truemm]{geometry}
\usepackage{url}

\theoremstyle{definition}
\newtheorem{thm}{Theorem}[section]

\newtheorem{definition}[thm]{Definition}
\newtheorem{lem}[thm]{Lemma}
\newtheorem{cor}[thm]{Corollary}
\newtheorem{rei}[thm]{Example}

\numberwithin{equation}{section}

\newcommand{\diag}{{\rm diag}}

\newcommand{\Mat}{{\rm Mat}}

\begin{document}
\title[]{Generalized group determinant gives a necessary and sufficient condition for a subset of a finite group to be a subgroup}
\author[Naoya Yamaguchi and Yuka Yamaguchi]{Naoya Yamaguchi and Yuka Yamaguchi}
\date{\today}

\maketitle

\begin{abstract}
We generalize the concept of the group determinant and prove a necessary and sufficient novel condition for a subset to be a subgroup.
This development is based on the group determinant work by Edward Formanek, David Sibley, and Richard Mansfield, where they show that two groups with the same group determinant are isomorphic. 
The derived condition leads to a generalization of this result. 
\end{abstract}

\section{Introduction}
We generalize the concept of the group determinant and prove a necessary and sufficient novel condition for a subset to be a subgroup.
The study of the group determinant by Edward Formanek, David Sibley, and Richard Mansfield shows that two groups with the same group determinant are isomorphic, and we generalize this result. 

Let $G$ be a finite group. 
The group determinant $\Theta(G)$ of $G$ was defined by Julius Wilhelm Richard Dedekind (see, e.g.,~\cite[p.~150]{Hawkins1971}, \cite[p.~224]{van2013history}). 
Ferdinand Georg Frobenius created the character theory of groups by studying the irreducible factorization of the group determinant (see, \cite{Frobenius1968gruppen}, \cite{Frobenius1968gruppencharaktere}). 
For the history on this theory, see, e.g.,~\cite{Curtis2001}, \cite{Hawkins1971}, \cite{Hawkins1972}, \cite{Hawkins1974}, \cite{lam1998representations}, and~\cite{van2013history}. 
On the group determinant, which has played a very important role in the history of mathematics, 
Edward Formanek and David Sibley gave the following theorem in 1991. 

\begin{thm}[{Special case of \cite[Theorem~$5$]{Formanek_1991}}]\label{thm:1.3}
Let $G$ and $H$ be finite groups, $K$ be a field whose characteristic does not divide $|G|$, 
and $\varphi: G \to H$ be a bijection such that $\varphi(e) = e'$, 
where $e$ and $e'$ are the unit elements of $G$ and $H$, respectively. 
Suppose that $\hat{\varphi} (\Theta(G)) = \Theta(H)$. 
Then, map $\varphi$ is either a group isomorphism or a group anti-isomorphism. 
In any case, $G$ is isomorphic to $H$, 
since any group is anti-isomorphic to itself. 
\end{thm} 

This theorem shows that two groups with the same group determinant are isomorphic. 
Shortly afterwards, Richard Mansfield gave a method to calculate a group multiplication table using the group determinant \cite{Richard}. 
Hence, Formanek, Sibley, and Mansfield showed that a group determinant determines the group. 
We extend their result by generalizing the concept of the group determinant. 

First, we introduce the generalized group determinant. 
Let $S$ be a subset of $G$,  
$x_{s}$ with $s \in S$ be independent commuting variables, and 
$x_{g} = 0$ for $g \in G \setminus S$. 
%Let $\Mat(r, A)$ be the set of all $r \times r$ matrices with elements in $A$. 
In addition, let $K$ be a field and $K [x_{s}]$ be the polynomial ring in $[x_{s}]$ over $K$, where $[x_{s}] = \{ x_{s} \mid s \in S \}$. 
We define a generalized group determinant of $G$ to $S$ in $[x_{s}]$ as 
$$
\Theta_{G}(S ; [x_{s}]) := \det{\left( x_{g h^{-1}} \right)_{g, h \in G}} \in K [x_{s}]. 
$$
If $S = G$, then we have $\Theta_{G}(S; [x_{s}]) = \Theta(G)$. 

Let $\left| A \right|$ be the cardinality of a set $A$.  
The main result of this paper is the following theorem, which asserts that 
the generalized group determinant gives a necessary and sufficient condition for a subset of a finite group to be a subgroup. 
%having the unit element.   
Note that if $\varphi : T \rightarrow S$ is a bijection, then $\varphi$ induces the $\mathbb{C}$-algebra isomorphism
$$
\hat{\varphi} : \mathbb{C} \left\{ x_{t} \mid t \in T \right\} \ni x_{t} \mapsto x_{\varphi(t)} \in \mathbb{C} \left\{ x_{s} \mid s \in S \right\}, 
$$
where $x_{t}$ and $x_{s}$ are independent commuting variables. 
 
\begin{thm}[see Theorem~$\ref{thm:3.1}$ for proof]\label{thm:1.1}
Let $G$ be a finite group, $e$ be the unit element of $G$, and $S$ be a subset of $G$ such that $e \in S$ and $|S|$ divides $|G|$. 
Then, $S$ is a group if and only if there exist a group $H$ and a bijective map $\varphi : H \to S$ such that $\varphi(e') = e$ and 
$$
\Theta_{G}(S ; [x_{s}]) = \hat{\varphi} \left( \Theta(H)^{\frac{|G|}{|H|}} \right), 
$$
where $e'$ is the unit element of $H$. 
\end{thm}

The following theorem gives additional information on $S$ and $\varphi$ satisfying the condition of Theorem~$\ref{thm:1.1}$． 

\begin{thm}[see Theorem~$\ref{thm:4.4}$ for proof]\label{thm:1.2}
Let $G$ be a finite group, $e$ be the unit element of $G$, and $S$ be a subset of $G$ such that $e \in S$ and $|S|$ divides $|G|$. 
If there exist a group $H$ and a bijective map $\varphi : H \rightarrow S$
such that $\varphi(e') = e$ and 
$$
\Theta_{G}(S ; [x_{s}]) = \hat{\varphi} \left( \Theta(H)^{\frac{|G|}{|H|}} \right), 
$$
where $e'$ is the unit element of $H$, 
then $S$ is group isomorphic to $H$ and $\varphi$ is either a group isomorphism or a group anti-isomorphism. 
\end{thm} 

Theorem~$\ref{thm:1.2}$ contains the case where $K = \mathbb{C}$ of Theorem~$\ref{thm:1.3}$. 

The remainder of this paper is organized as follows. 
In Section~$2$, 
we define and provide examples of the generalized group determinant. 
We conclude that the generalized group determinant is a generalization of the group determinant. 
In Section~$3$, 
using the generalized group determinant, 
we give a necessary and sufficient condition for a subset of a finite group to be a subgroup, 
which is main result of this paper. 
Finally, 
we give additional information on subsets satisfying the necessary and sufficient condition in Section~$4$.

\section{Generalized group determinant}
In this section, we generalize the group determinant and give examples of this generalized group determinant. 
In the next section, we use the generalized group determinant to provide a necessary and sufficient condition for a subset of a finite group to be a subgroup. 

Let $G$ be a finite group, $S$ be a subset of $G$, 
$x_{s}$ with $s \in S$ be independent commuting variables, and $x_{g} = 0$ for $g \in G \setminus S$. 
In addition, let $\Mat(m, A)$ be the set of all $m \times m$ matrices with elements in a set $A$ and $\left| B \right|$ be the cardinality of a set $B$.  
We define a matrix $M_{G}(S; [x_{s}])$ as 
\begin{align*}
M_{G}(S; [x_{s}]) := \left( x_{g h^{-1}} \right)_{g, h \in G} \in \Mat(|G|, [x_{s}]), 
\end{align*}
where  $[x_{s}] = \{ x_{s} \mid s \in S \}$. 

Let $K$ be a field and $K [x_{s}]$ be the polynomial ring in $[x_{s}]$ over $K$. 
\begin{definition}\label{def:2.1}
We define a generalized group determinant of $G$ to $S$ in $[x_{s}]$ as 
$$
\Theta_{G}(S ; [x_{s}]) := \det{M_{G}(S; [x_{s}])} \in K [x_{s}]. 
$$
For conciseness, when $S = G$, we write $\Theta_{G}(S; [x_{s}])$ as $\Theta(G; [x_{s}])$ or simply $\Theta(G)$. 
\end{definition}

The polynomial $\Theta(G)$ is called the group determinant of $G$ (see, e.g.,~\cite[p.~366]{conrad1998origin}, \cite[p.~38]{Frobenius1968gruppen}, \cite[p.~142]{Hawkins1971}, \cite[p.~299]{johnson1991}, \cite[p.~224]{van2013history}, \cite[p.~7]{Yamaguchi2017}). 
It follows from the above definition that $\Theta_{G}(S ; [x_{s}])$ is a homogeneous polynomial of degree $|G|$ in $[x_{s}]$.  

In general, the matrix $M_{G}(S; [x_{s}])$ is covariant under a change of numbering to the elements of $G$. 
However, the generalized group determinant $\Theta_{G}(S ; [x_{s}])$ is invariant. 
We illustrate the generalized group determinant in the following example. 

\begin{rei}\label{rei:2.2}
Let $G = \mathbb{Z} / 3 \mathbb{Z} = \left\{ \overline{1}, \overline{2}, \overline{3} \right\} $. 
%and let $g_{i} = \overline{i}$ for any $i \in \{ 1, 2, 3 \}$. 
For conciseness, we write $x_{\overline{i}}$ as $x_{i}$ for any $i \in \{ 1, 2, 3 \}$. 
If $S = G$, 
we have 
\begin{align*}
\Theta(G) = \det{
\begin{pmatrix}
x_{3} & x_{2} & x_{1} \\ 
x_{1} & x_{3} & x_{2} \\ 
x_{2} & x_{1} & x_{3}
\end{pmatrix}
} 
= x_{1}^{3} + x_{2}^{3} + x_{3}^{3} -  3 x_{1} x_{2} x_{3}. 
\end{align*}
If $S = \left\{ \overline{3} \right\}$, we have 
\begin{align*}
\Theta_{G}(S ; [x_{s}]) = \det{
\begin{pmatrix}
x_{3} & 0 & 0 \\ 
0 & x_{3} & 0 \\ 
0 & 0 & x_{3}
\end{pmatrix}
} 
= x_{3}^{3}. 
\end{align*}
\end{rei} 

When $S = G$, the matrix $M_{G}(S; [x_{s}])$ is called the group matrix of $G$ (see, e.g.,~\cite[p.~366]{conrad1998origin}, \cite[p.~649]{Formanek_1991}, \cite[p.~276]{Frobenius1968gruppen2}, \cite[p.~4]{Andy}, \cite[p.~299]{johnson1991}). 
Below, we assume that $K = \mathbb{C}$. 
The group matrix is a matrix form of an element $\alpha := \sum_{g \in G} x_{g} g$ in the group algebra $\mathbb{C} G$, 
where we assume that $x_{g}$ is a complex number for any $g \in G$. 
That is, the map $L : \mathbb{C}G \rightarrow \Mat(|G|, \mathbb{C})$ given by $\alpha \mapsto M_{G}(G; [x_{g}])$ is a ring homomorphism. 
In particular, $L$ is the regular representation of $G$ (see, e.g.,~\cite[p.~143, p.~161]{Hawkins1971}, \cite[p.~16]{Yamaguchi2018transfer}).

\section{Necessary and sufficient condition for a subset of a finite group to be a subgroup}

In this section, using the generalized group determinant, we give a necessary and sufficient condition for a subset of a finite group to be a subgroup. 
%having the unit element. 

The following theorem gives the necessary and sufficient condition. 
Note that if $\varphi : T \rightarrow S$ is a bijective map, then $\varphi$ induces the $\mathbb{C}$-algebra isomorphism
$$
\hat{\varphi} : \mathbb{C} \left\{ x_{t} \mid t \in T \right\} \ni x_{t} \mapsto x_{\varphi(t)} \in \mathbb{C} \left\{ x_{s} \mid s \in S \right\}, 
$$
where $x_{t}$ and $x_{s}$ are independent commuting variables. 

\begin{thm}[Theorem~\ref{thm:1.1}]\label{thm:3.1}
Let $G$ be a finite group, $e$ be the unit element of $G$, and $S$ be a subset of $G$ such that $e \in S$ and $|S|$ divides $|G|$. 
Then, $S$ is a group if and only if there exist a group $H$ and a bijective map $\varphi : H \to S$ 
such that $\varphi(e') = e$ and 
$$
\Theta_{G}(S ; [x_{s}]) = \hat{\varphi} \left( \Theta(H)^{\frac{|G|}{|H|}} \right) = \Theta \left( H; \left\{ x_{\varphi(h)} \mid h \in H\right\} \right)^{\frac{|G|}{|H|}}, 
$$
where $e'$ is the unit element of $H$. 
\end{thm} 

First, we give a necessary condition for a subset of a finite group to be a subgroup.
%having the unit element. 

\begin{lem}\label{lem:3.2}
If $S$ is a subgroup of $G$, then 
$$
\Theta_{G}(S ; [x_{s}]) = \Theta(S)^{\frac{|G|}{|S|}}. 
$$
\end{lem} 
\begin{proof}
Let $l = |G|/|S|$ and $S \backslash G = \{ Sg'_{1}, Sg'_{2}, \ldots, Sg'_{l} \}$ be a right coset of $S$ in $G$. 
We put $g_{i} = s_{p} g'_{q} \in G$, where $i = |S| (q - 1) + p$ for $1 \leq p \leq |S|$ and $1 \leq q \leq l$. 
With respect to this ordering of the elements of $G$, 
observe that $M_{G}(S; [x_{s}]) = \diag(M_{S}(S; [x_{s}]), M_{S}(S; [x_{s}]), \ldots, M_{S}(S; [x_{s}]))$. 
So, we have 
\begin{align*}
\Theta_{G}(S ; [x_{s}]) = \det \left\{ \diag(M_{S}(S; [x_{s}]), M_{S}(S; [x_{s}]), \ldots, M_{S}(S; [x_{s}])) \right\} = \Theta(S)^{\frac{|G|}{|S|}}. 
\end{align*}
\end{proof}

Let $e$ be the unit element of $G$. 
To give a sufficient condition for a subset of a finite group to be a subgroup, 
%having the unit element, 
we use the following lemma, 
which is known in terms of the monomials of group determinants. 

\begin{lem}[{\cite[Lemmas~1--3]{Richard}}]\label{lem:3.3}
Let $n = |G|$. 
The following hold: 
\begin{enumerate}
\item If the monomial $x_{a_{1}} x_{a_{2}} \cdots x_{a_{n}}$ occurs in $\Theta(G)$, the $a_{i}$ can be ordered such that their product is $e$; 
\item If $a b = e$, the monomial $x_{e}^{n-2} x_{a} x_{b}$ occurs in $\Theta(G)$; 
\item If $a b c = e$, the monomial $x_{e}^{n-3} x_{a} x_{b} x_{c}$ occurs in $\Theta(G)$; 
\item If none of $a$, $b$, $c$ is $e$ and the monomial $x_{e}^{n-3} x_{a} x_{b} x_{c}$ occurs in $\Theta(G)$, 
the coefficient of the monomial is {\rm (i)} $n / 3$ if $a = b = c$; {\rm (ii)} $n$ if two of $a$, $b$, $c$ are equal; 
{\rm (iii)} $n$ if no two of them are equal and $a b \neq b a$; {\rm (iv)} $2n$ if no two of them are equal and $a b = b a$. 
$($Note that if $a b c = e$, then $a b = b a$ if and only if $a$, $b$ and $c$ are commutative.$)$
\end{enumerate}
\end{lem} 
Here, we say that a monomial occurs in a polynomial if the monomial is not canceled after combining like terms. 
There is a mistake in the last sentence of the proof of \cite[Lemma~3]{Richard}. 
It says that the coefficient of $x_{e}^{n-3} x_{a} x_{b} x_{c}$ is $n$ or $2n$, but the coefficient is $n / 3$ when $a = b = c$. 
The following is a proof of (4) of Lemma~$\ref{lem:3.3}$. 
\begin{proof}
If $a = b = c$, $a^{3} = e$ from (1) of Lemma~$\ref{lem:3.3}$. 
This implies that $S = \left\{e, a, a^{2} \right\}$ is a subgroup of $G$. 
Therefore, from Lemma~$\ref{lem:3.2}$, we have 
$$
\Theta_{G}(S ; [x_{s}]) = \Theta(S)^{\frac{|G|}{|S|}} = \left( x_{e}^{3} + x_{a}^{3} + x_{a^{2}}^{3} - 3 x_{e} x_{a} x_{a^{2}}\right)^{\frac{n}{3}}. 
$$
From this, case (i) is proved 
since the coefficients of $x_{e}^{n-3} x_{a}^{3}$ in $\Theta(G)$ and $\Theta_{G}(S ; [x_{s}])$ are equal. 
For any $g_{i} \in G$ with $i = 1, 2, \ldots, n$, there is only one pair $(g_{j}, g_{k})$ of elements of $G$ such that 
$(g_{i} g_{j}^{-1}, g_{j} g_{k}^{-1}, g_{k} g_{i}^{-1}) = (a, b, c)$. 
If $a$, $b$, and $c$ have the relation as in cases (ii) or (iii), any permutation $\sigma \in S_{n}$ giving the desired monomial 
(i.e., satisfying $\prod_{l = 1}^{n} x_{g_{l} g_{\sigma(l)}^{-1}} = x_{e}^{n-3} x_{a} x_{b} x_{c}$) 
must be $\sigma = (i\, j\, k)$ for some $i \neq j, k$. 
This proves cases (ii) and (iii). 
If no two of $a$, $b$, and $c$ are equal and $a b = b a$, any permutation $\sigma \in S_{n}$ giving the desired monomial  
must be $\sigma = (i\, j\, k)$ or $(i\, j\, k')$, where $g_{k'} = g_{i} g_{k}^{-1} g_{j}$. 
This proves case (iv).
\end{proof}

From (1) and (4) of Lemma~$\ref{lem:3.3}$, we obtain the following corollary. 

\begin{cor}\label{cor:3.4}
Let $n = |G|$. 
The following hold: 
\begin{enumerate}
\item If the monomial $x_{a_{1}} x_{a_{2}} \cdots x_{a_{n}}$ occurs in $\Theta_{G}(S ; [x_{s}])$, 
the $a_{i}$ can be ordered such that their product is $e$; 
\item If none of $a$, $b$, $c$ is $e$ and the monomial $x_{e}^{n-3} x_{a} x_{b} x_{c}$ occurs in $\Theta_{G}(S ; [x_{s}])$, 
the coefficient of the monomial is {\rm (i)} $n / 3$ if $a = b = c$; {\rm (ii)} $n$ if two of $a$, $b$, $c$ are equal; 
{\rm (iii)} $n$ if no two of them are equal and $a b \neq b a$; {\rm (iv)} $2n$ if no two of them are equal and $a b = b a$. 
\end{enumerate}
\end{cor}

We now give a sufficient condition for a subset of a finite group to be a subgroup. 
%having the unit element. 
\begin{lem}\label{lem:3.5}
Let $S$ be a subset of $G$ such that $e \in S$ and $|S|$ divides $|G|$. 
If there exist a group $H$ and a bijection $\varphi : H \rightarrow S$ such that $\varphi(e') = e$ and 
$$
\Theta_{G}(S ; [x_{s}]) = \hat{\varphi} \left( \Theta(H)^{\frac{|G|}{|H|}} \right) = \Theta \left( H; \left\{ x_{\varphi(h)} \mid h \in H\right\} \right)^{\frac{|G|}{|H|}}, 
$$
where $e'$ is the unit element of $H$, then $S$ is a group. 
\end{lem} 
\begin{proof}
For all $s, s' \in S$, we show that $s^{-1} \in S$ and $s s' \in S$. 
Let $h = \varphi^{-1}(s)$ and $h' = \varphi^{-1}(s')$. 
First, we prove $s^{-1} \in S$. 
From (2) of Lemma~$\ref{lem:3.3}$, 
the monomials $x_{e'}^{|H|}$ and $x_{e'}^{|H|-2} x_{h} x_{h^{-1}}$ occur in $\Theta(H)$. 
Hence, the monomial $x_{e'}^{|G|-2} x_{h} x_{h^{-1}}$ occurs in $\Theta(H)^{\frac{|G|}{|H|}}$. 
Applying $\hat{\varphi}$ to this monomial, we find that the monomial $x_{e}^{|G|-2} x_{s} x_{\varphi(h^{-1})}$ occurs in $\Theta_{G}(S ; [x_{s}])$. 
Therefore, from (1) of Corollary~$\ref{cor:3.4}$, we have $s^{-1} = \varphi(h^{-1}) \in S$. 
Next, we prove $s s' \in S$ under the assumption that $s \neq e$ and $s' \neq e$ (because it is obvious for $s = e$ or $s' = e$).
If $s \neq e$ and $s' \neq e$, then $h \neq e$ and $h' \neq e$. 
Note that if $(s s')^{-1} \in S$, then $s s' \in S$ from the above result. 
From (3) of Lemma~$\ref{lem:3.3}$, the monomials $x_{e'}^{|H|}$, 
$x_{e'}^{|H|-3} x_{h} x_{h'} x_{{(h h')}^{-1}}$ and $x_{e'}^{|H|-3} x_{h} x_{h'} x_{{(h' h)}^{-1}}$ occur in $\Theta(H)$. 
Hence, the monomials $x_{e'}^{|G|-3} x_{h} x_{h'} x_{(h h')^{-1}}$ and $x_{e'}^{|G|-3} x_{h} x_{h'} x_{(h' h)^{-1}}$ occur in $\Theta(H)^{\frac{|G|}{|H|}}$. 
Applying $\hat{\varphi}$ to these monomials, 
we find that the monomials $x_{e}^{|G|-3} x_{s} x_{s'} x_{\varphi((h h')^{-1})}$ and $x_{e}^{|G|-3} x_{s} x_{s'} x_{\varphi((h' h)^{-1})}$ occur in $\Theta_{G}(S ; [x_{s}])$. 
We classify the relation between $h$ and $h'$ into the four following cases: 
\begin{enumerate}
\item $h h' \neq h' h$; 
\item $h = h'$; 
\item $h h' = h' h$, $h \neq h'$ and either $h$ or $h'$ is equal to $(h h')^{-1}$; 
\item $h h' = h' h$ and no two of $h$, $h'$, $(h h')^{-1}$ are equal. 
\end{enumerate}
If case~$(1)$ holds, 
then, from (1) of Corollary~$\ref{cor:3.4}$, we have $(s s')^{-1} = \varphi((h h')^{-1}) \in S$ or $(s s')^{-1} = \varphi((h' h)^{-1}) \in S$. 
If case~$(2)$ holds, 
then, from (1) of Corollary~$\ref{cor:3.4}$, we have $(s s')^{-1} = \varphi((h h)^{-1}) \in S$. 
If case~$(3)$ holds, then, from (1) of Corollary~$\ref{cor:3.4}$, we have $s^{2} s' = e$ or $s (s')^{2} = e$. 
Hence, $s s' = s^{-1} \in S$ or $s s' = s'^{-1} \in S$. 
If case~$(4)$ holds, then, from (4) of Lemma~$\ref{lem:3.3}$, the coefficient of the monomial $x_{e'}^{|H|-3} x_{h} x_{h'} x_{{(h h')}^{-1}}$ in $\Theta(H)$ is $2 |H|$. 
Therefore, the coefficient of the monomial $x_{e'}^{|G|-3} x_{h} x_{h} x_{(h h)^{-1}}$ in $\Theta(H)^{\frac{|G|}{|H|}}$ is $2 |H| \times \frac{|G|}{|H|} = 2 |G|$. 
From this, the coefficient of the monomial $x_{e}^{|G|-3} x_{s} x_{s'} x_{\varphi((h h')^{-1})}$ in $\Theta_{G}(S ; [x_{s}])$ is $2 |G|$. 
Hence, from (2) of Corollary~$\ref{cor:3.4}$, we have $s s' = s' s$, 
and from (1) of Corollary~$\ref{cor:3.4}$, we have $(s s')^{-1} = \varphi((h h')^{-1}) \in S$. 
\end{proof}

\section{Additional information on the main theorem}

The following theorem gives additional information on $S$ and $\varphi$ satisfying the condition of Theorem~$\ref{thm:3.1}$. 
\begin{thm}[Theorem~\ref{thm:1.2}]\label{thm:4.4}
Let $G$ be a finite group, $e$ be the unit element of $G$, and $S$ be a subset of $G$ such that $e \in S$ and $|S|$ divides $|G|$. 
If there exists a group $H$ and a bijective map $\varphi : H \rightarrow S$
such that $\varphi(e') = e$ and 
$$
\Theta_{G}(S ; [x_{s}]) = \hat{\varphi} \left( \Theta(H)^{\frac{|G|}{|H|}} \right) = \Theta\left(H; \left\{ x_{\varphi(h)} \mid h \in H\right\} \right)^{\frac{|G|}{|H|}}, 
$$
where $e'$ is the unit element of $H$, 
then $\varphi$ is either a group isomorphism or a group anti-isomorphism, and $S$ is group isomorphic to $H$. 
\end{thm}

Let $\mathcal{U} \left( R \right)$ be the units of a ring $R$. 
To prove Theorem~$\ref{thm:4.4}$, we use the following theorem and lemma. 

\begin{thm}[{\cite[Theorem~6]{Andy}}]\label{thm:4.3}
Let $G$ and $H$ be finite groups with unit elements $e$ and $e'$, respectively. 
Let $\psi : G \to H$ be a bijection that induces a vector space isomorphism $\widetilde{\psi} : \mathbb{C} G \to \mathbb{C} H$. 
Then the following are equivalent: 
\begin{enumerate}
\item map $\psi : G \to H$ is an isomorphism or an anti-isomorphism; 
\item $\widetilde{\psi} \left( \mathcal{U} \left( \mathbb{C} G \right) \right) = \mathcal{U} \left( \mathbb{C} H \right)$ and $\psi(e) = e'$. 
\end{enumerate}
\end{thm}

\begin{lem}\label{lem:}
We regard the variable $x_{g}$ as a complex number for any $g \in G$. 
Then, the element $\sum_{g \in G} x_{g} g$ is invertible in $\mathbb{C} G$ if and only if $\Theta(G; \{x_{g} \mid g \in G \})$ is invertible in $\mathbb{C}$. 
\end{lem} 
\begin{proof}
Let $\alpha := \sum_{g \in G} x_{g} g$ in the group algebra $\mathbb{C} G$, 
where we assume that $x_{g}$ is a complex number for any $g \in G$. 
Then, as mentioned at the end of Section~2, 
the map $L : \mathbb{C}G \rightarrow \Mat(|G|, \mathbb{C})$ given by $\alpha \mapsto M_{G}(S; [x_{s}])$ is the regular representation from $\mathbb{C} G$ to $\Mat(|G|, \mathbb{C})$. 
Therefore, since the group algebra $\mathbb{C}G$ is isomorphic as a $\mathbb{C}$-algebra to a direct product of matrix algebras over $\mathbb{C}$ (see, e.g.,~\cite[Theorem~5.5.6]{benjamin2013}), 
the element $\alpha$ is invertible in $\mathbb{C}G$ if and only if $\det{L(\alpha)} \neq 0$. 
\end{proof}

The following is a proof of Theorem~$\ref{thm:4.4}$. 
\begin{proof}
From Lemma~$\ref{lem:3.5}$, $S$ is a group. 
Hence, from Lemma~$\ref{lem:3.2}$, we have 
$\Theta(S)^{\frac{|G|}{|S|}} = \Theta_{G}(S ; [x_{s}]) = \hat{\varphi} \left( \Theta(H)^{\frac{|G|}{|H|}} \right) = \Theta \left(H; \left\{ x_{\varphi(h)} \mid h \in H\right\} \right)^{\frac{|G|}{|H|}}$. 
We regard the variable $x_{s}$ as a complex number for any $s \in S$. 
Then, we have: 
\begin{enumerate}
\item[(1)] $\Theta(S) \neq 0$ if and only if $\Theta \left(H; \left\{ x_{\varphi(h)} \mid h \in H\right\} \right) \neq 0$. 
\end{enumerate}
In addition, from Lemma~$\ref{lem:}$, we have: 
\begin{enumerate}
\item[(2)] $\Theta(S) \neq 0$ if and only if $\sum_{s \in S} x_{s} s$ is invertible; 
\item[(3)] $\Theta \left(H; \left\{ x_{\varphi(h)} \mid h \in H\right\} \right) \neq 0$ if and only if $\sum_{h \in H} x_{\varphi(h)} h$ is invertible. 
\end{enumerate}
From (1)--(3), $\sum_{s \in S} x_{s} s$ is invertible if and only if $\sum_{h \in H} x_{\varphi(h)} h$ is invertible. 
On the other hand, we have 
\begin{align*}
\widetilde{\varphi^{-1}} \left( \sum_{s \in S} x_{s} s \right) = \sum_{s \in S} x_{s} \varphi^{-1}(s) = \sum_{h \in H} x_{\varphi(h)} h, 
\end{align*}
where $\widetilde{\varphi^{-1}}$ is the $\mathbb{C}$-linear map induced by $\varphi^{-1}$. 
Therefore, we have $\widetilde{\varphi^{-1}} (\mathcal{U}(\mathbb{C} S)) = \mathcal{U}(\mathbb{C} H)$. 
From Theorem~$\ref{thm:4.3}$, $\varphi$ is a group isomorphism or a group anti-isomorphism, and $S$ is group isomorphic to $H$.  
\end{proof}

Theorem~$\ref{thm:4.4}$ contains the case where $K = \mathbb{C}$ of the following theorem. 

\begin{thm}[{Special case of \cite[Theorem~$5$]{Formanek_1991}}]\label{thm:4.5}
Let $G$ and $H$ be finite groups with unit elements $e$ and $e'$, respectively. 
Let $K$ be a field whose characteristic does not divide $|G|$ and $\varphi: G \to H$ be a bijection such that $\varphi(e) = e'$. 
Suppose that $\hat{\varphi} \left( \Theta(G; \{ x_{g} \mid g \in G \} ) \right) = \Theta(H; \{ x_{h} \mid h \in H \})$. 
Then, the map $\varphi$ is either a group isomorphism or a group anti-isomorphism. 
In any case, $G$ is isomorphic to $H$, 
since any group is anti-isomorphic to itself. 
\end{thm}

\clearpage

\thanks{\bf{Acknowledgments}}
We are deeply grateful to professor Hiroyuki Ochiai, who provided helpful comments and suggestions on this work.

\bibliographystyle{amsplain}
\bibliography{reference}

%\bibliography{ref}

%\begin{thebibliography}{99}

%\bibitem{conrad1998origin}
%CONRAD, Keith. On the origin of representation theory. {\em ENSEIGNEMENT MATHEMATIQUE}, 1998, 44: 361--392. 

%\bibitem{} 
%FORMANEK, Edward; SIBLEY, David. The group determinant determines the group. {\em Proceedings of the American Mathematical Society}, 1991, 112.3: 649--656. 

%\bibitem{benjamin} 
%STEINBERG, Benjamin. {\em Representation theory of finite groups: an introductory approach}. Springer Science \& Business Media, 2011. 

%\bibitem{johnson1991group} 
%JOHNSON, Kenneth W. On the group determinant. In: {\em Mathematical Proceedings of the Cambridge Philosophical Society}. Cambridge University Press, 1991, 109.02: 299--311.

%\bibitem{Lam} 
%LAM, T. Y. Representations of finite groups: A hundred years, part II. {\em Notices of the AMS}, 1998, 45.4.

%\bibitem{van} 
%VAN DER WAERDEN, Bartel L. {\em A history of algebra: from al-Khw{\=a}rizm{\=\i} to Emmy Noether}. Springer Science \& Business Media, 2013.

%\bibitem{}
%YAMAGUCHI, Naoya. Factorization of group determinant in some group algebras. {\em arXiv preprint arXiv:1405.1900}, 2014.

%\end{thebibliography}

\end{document}